\documentclass{proc-l}

\newcommand{\DF}[2]{{\displaystyle\frac{#1}{#2}}}
\newtheorem{theorem}{Theorem}[section]
\newtheorem{lemma}[theorem]{Lemma}

\theoremstyle{definition}
\newtheorem{definition}[theorem]{Definition}
\newtheorem{example}[theorem]{Example}

\theoremstyle{remark}

\numberwithin{equation}{section}

\begin{document}

\title{Symplectic rigidity of real bidisc}



\author{Yat-Sen Wong}
\address{1409 W. Green Street, Urbana IL 61801}
\curraddr{}
\email{sfxcwys@gmail.com}


\date{Nov 2, 2014}

\dedicatory{}

\commby{}

\begin{abstract}
Let $\mathbb{D}$ be the unit disc in $\mathbb{C}$, then $\mathbb{D}^n(r)$ is the complex or symplectic $n$-discs of radius $r$. Let $z_j = x_j+iy_j\in\mathbb{C}, j=1,2$ and $\mathbb{D}_{\mathbb{R}}^2=\{(z_1,z_2) : |x_1|^2+|x_2|^2<1,|y_1|^2+|y_2|^2<1\}$ be the real bidisc. In this paper we will prove the following two theorems:
\begin{enumerate}
\item If $T\in O(4)$ is an orthogonal transformation on $\mathbb{R}^4$, then $T(\mathbb{D}^2)$ is symplectomorphic to $\mathbb{D}^2$ w.r.t. the standard symplectic form on $\mathbb{R}^4$ if and only if $T$ is unitary or conjugate to unitary.
\item For $r\geq 1$ and $n\geq 2$, $\mathbb{D}_{\mathbb{R}}^2\times \mathbb{D}^{n-2}(r)$ and $\mathbb{D}^2\times \mathbb{D}^{n-2}(r)$ are not symplectomorphic w.r.t. the standard symplectic form on $\mathbb{C}^n$.
\end{enumerate}
\end{abstract}

\maketitle


\bibliographystyle{amsplain}



\section{Introduction}
Let $x_1,y_1,\dots, x_n,y_n$ be the standard coordinates on the $2n$-dimensional Euclidean space $\mathbb{R}^{2n}\cong\mathbb{C}^n$, the standard symplectc form on the space is given by $dx_1\wedge dy_1+\dots+dx_n\wedge dy_n$. All symplectic embeddings considered in this paper will be with respect to the standard symplectic form, unless otherwise specified. Define the standard disc in $\mathbb{C}$ of radius $R$ by $\mathbb{D}(R)=\{z\in\mathbb{C} : |z| < R\}$, also define $\mathbb{D}_{\mathbb{R}}^2(r)=\{(z_1,z_2)\in\mathbb{C}^2 : |x_1|^2+|x_2|^2<r,|y_1|^2+|y_2|^2<r\}$ be the real bidisc of radius $r$. We denote $\mathbb{D}(1)$ by $\mathbb{D}$ and $\mathbb{D}_{\mathbb{R}}^2(1)$ by $\mathbb{D}_{\mathbb{R}}^2$. We denote by $\mathbb{B}^{2n}(a)$ the $2n$-dimensional Euclidean ball of radius $a$ in $\mathbb{R}^{2n}$.

It is proved by Sukhov and Tumanov \cite{AT_Symp_Nonsqueeze} that the real bi-disc $\mathbb{D}_{\mathbb{R}}^2$ cannot be symplectically embedded into the complex cylinder $\mathbb{D}\times\mathbb{C}$. If we consider the real bidisc as obtained from a non-holomorphic change of coordinates
$$
T_0:(x_1, y_1, x_2, y_2)\mapsto (x_1, x_2, y_1, y_2)
$$
of $\mathbb{D}^2$, then the result of Sukhov and Tumanov shows that $T_0(\mathbb{D}^2)$ is not symplectomorphic to $\mathbb{D}^2$ itself. The first main result of this paper generalizes this observation: if $T\in O(4)$ is any orthogonal transformation on $\mathbb{R}^4=\mathbb{C}^2$, then $T(\mathbb{D}^2)$ is symplectomorphic to $\mathbb{D}^2$ if and only if $T$ is unitary or conjugate to unitary. We will give a more precise statement in Section~\ref{section2}.

The second result of this paper considers a high dimensional analogy of the previous result. We will show that for $r\geq 1$ and $n\geq 2$, $\mathbb{D}_{\mathbb{R}}^2\times\mathbb{D}^{n-2}(r)$ is not symplectomorphic to $\mathbb{D}^{2}\times\mathbb{D}^{n-2}(r)$.

The first striking result on symplectic rigidity was obtained by Gromov \cite{Gromov}, which states that one can symplectically embed a sphere into a cylinder only if the radius of the sphere is less than or equal to the radius of the cylinder. Following Gromov's work, many results on symplectic rigidity were obtained for various domains. For example, McDuff \cite{McDuff_ellipsoid} studied when a 4-dimensional ellipsoid can be symplectically embedded in a ball; Guth \cite{Guth} gave an asymptotic result on when a polydisc $\mathbb{D}(r_1)\times\cdots\times \mathbb{D}(r_n)$ can be symplectically embedded into another. Our results have the same spirit, but we deal with essentially different domains: real bidisc and its modifications.

There are a lot of open problems concerning symplectic rigidity, for instance it is not known that whether $\mathbb{D}_{\mathbb{R}}^2 \times \mathbb{D}(r)$ is symplectomorphic to $\mathbb{D}^2\times \mathbb{D}(r)$ when $r < 1$. The results in this paper only show that such symplectomorphism does not exist when $r \geq 1$. Another interesting open problem is to characterize when two given polydiscs are symplectomorphic.

\section{$J$-holomorphic discs and symplectic manifolds}

In this section we will recall some basic properties of $J$-holomorphic discs and symplectic manifolds.

\begin{definition}
A smooth map $\phi:(M,J)\rightarrow (M',J')$ from one almost complex manifold to another is said to be $(J, J')$-holomorphic if its derivative $d\phi$ is complex linear, that is
\begin{equation}\label{eq:CauchyRiemann1}
d\phi\circ J = J'\circ d\phi.
\end{equation}
\end{definition}

Denote by $J_{\mbox{st}}$ the standard complex structure of $\mathbb{C}^n$. A $J$-holomorphic disc or pseudo-holomorphic disc is a $(J_{\mbox{st}},J)$-holomorphic map

$$
u:\mathbb{D}\rightarrow M
$$
from $\mathbb{D}$ to an almost complex manifold $(M,J)$.

In local coordinates $z\in\mathbb{C}^n$, an almost complex structure $J$ is represented by a $\mathbb{R}$-linear operator $J(z):\mathbb{C}^n\rightarrow\mathbb{C}^n$ such that $J(z)^2 = -I$, where $I$ is the identity map. Now the Cauchy-Riemann equations (\ref{eq:CauchyRiemann1}) for a $J$-holomorphic disc $z:\mathbb{D}\rightarrow\mathbb{C}^n$ can be written in the form

$$
z_{\eta}=J(z)z_{\xi}, \zeta=\xi+i\eta\in\mathbb{D}.
$$

We represent $J$ by a complex $n\times n$ matrix function $A = A(z)$ and obtain the equivalent equations

\begin{equation} \label{eq:CauchyRiemann2}
z_{\overline{\zeta}} = A(z)\overline{z}_{\overline{\zeta}}, \zeta\in\mathbb{D}.
\end{equation}

We recall the relation between $J$ and $A$ for fixed $z$. let $J:\mathbb{C}^n\rightarrow\mathbb{C}^n$ be a $\mathbb{R}$-linear map so that det$(J_{\mbox{st}}+J)\neq 0$, where $J_{\mbox{st}}v = iv.$ Set
$$
Q = (J_{\mbox{st}}+J)^{-1}(J_{\mbox{st}}-J).
$$

\begin{lemma}
(\cite{Audin})
\label{thm:lemma1}
$J^2 = -I$ if and only if $QJ_{\mbox{st}} + J_{\mbox{st}}Q = 0$.
\end{lemma}

Notice that $QJ_{\mbox{st}} + J_{\mbox{st}}Q = 0$ is equivalent to $Q$ being a complex anti-linear operator. Therefore Lemma \ref{thm:lemma1} implies that there is a unique matrix $A\in$Mat$(n, \mathbb{C})$ such that
$$
Av = Q\overline{v}, v\in\mathbb{C}^n.
$$



Let $M$ be a smooth manifold of real dimension $2n$. A closed non-degenerate exterior 2-form $\omega$ on $M$ is called a symplectic form on $M$. A couple $(M,\omega)$ is called a symplectic manifold. A basic example is $M = \mathbb{C}^n$ with the coordinates $z_j = x_j + iy_j, j=1,\dots, n$. The form $\omega_{\mbox{st}} = \sum^n_{j=1} dx_j\wedge dy_j = \frac{i}{2}\sum^n_{j=1}dz_j\wedge d\overline{z_j}$ is called the standard symplectic form on $\mathbb{C}^n$.

A symplectic form $\omega$ tames an almost complex structure $J$ on $M$ if $\omega(u,Ju)>0, \mbox{ for all } u\neq 0$. A basic example is $(M, \omega, J) = (\mathbb{C}^n, \omega_{\mbox{st}}, J_{\mbox{st}})$.

\begin{lemma}
(\cite{Audin})
Let $J$ be an almost complex structure on $\mathbb{C}^n$, then $J$ is tamed by $\omega_{\mbox{st}}$ if and only if the complex matrix $A$ of $J$ satisfies the condition
\begin{equation}
\label{eq:AT_direct_nonsqueeze_3}
\| A(z) \|<1, \mbox{ \emph{for all} } z\in\mathbb{C}^n.
\end{equation}
Here the matrix norm is induced by the Euclidean inner product, that is, $\|A\| = \mbox{\emph{max}}_{0\neq v\in\mathbb{R}^{2n}} |Av|_{\mathbb{R}^{2n}}/|v|_{\mathbb{R}^{2n}}$.
\end{lemma}

For a map $u:\mathbb{D}\rightarrow\mathbb{C}^n$, the (symplectic) area of $u$ is given by

\begin{equation}
\label{eq:sympArea}
\mbox{Area}(u) = \int_{\mathbb{D}}u^*\omega_{\mbox{st}}.
\end{equation}

If $J$ is $\omega_{\mbox{st}}$ tamed, we can consider the canonical Riemannian metric $g_{J}(X, Y)=\DF{1}{2}(\omega_{\mbox{st}}(X, JY) + \omega_{\mbox{st}}(Y, JX))$ determined by $J$ and $\omega_{\mbox{st}}$. Suppose $u$ is a $J$-holomorphic disc, then the symplectic area of $u$ coincides with the area induced by $g_{J}$; in particular, it coincides with the Euclidean area if $J=J_{\mbox{st}}$ (see \cite{Audin} for more details).


\section{Orthogonal transformation of complex bidisc}
\label{section2}
Let $T\in O(4)$ be an orthogonal transformation on $\mathbb{R}^4\cong\mathbb{C}^2$, let $\mathbb{D}^2 = \{ |z_j| < 1 \colon j = 1,2\}$ be the complex bidisc. In this section we will give a necessary and sufficient condition for $T(\mathbb{D}^2)$ to be symplectomorphic to $\mathbb{D}^2$ with respect to the standard symplectic form on $\mathbb{C}^2$.

First of all, we define the notion of holomorphic radius and state a theorem proved by A. Sukhov and A. Tumanov \cite{AT_Symp_Nonsqueeze} which provides a necessary condition on holomorphic radius for the existence of symplectic embedding.

\begin{definition}
Let $\Omega$ be a complex manifold. A closed set $A\subset\Omega$ is called a (complex) analytic set if it is, in a neighborhood of each of its points, the set of common zeros of a certain finite family of holomorphic functions. In this paper we only consider closed analytic sets.
\end{definition}

\begin{definition}
A point $p$ of an analytic set $A$ in a complex manifold $\Omega$ is called regular if there is a neighborhood $U$ in $\Omega$ containing $p$ such that $A\cap U$ is a complex submanifold of $U$. The complex dimension of this submanifold is said to be the dimension of $A$ at its regular point $p$, and is denoted by dim$_pA$. The set of all regular points of $A$ is denoted by reg$A$.
\end{definition}

It is a fundamental result of complex analytic sets that the set of all regular points of an analytic set $A$ is dense in $A$ (see, for example, \cite{Bishop}).

\begin{definition}
A purely $m$-dimensional analytic set $A$ is an analytic set such that for every $p\in\mbox{reg}A$, we have $\dim_pA=m$.
\end{definition}

\begin{definition}
Let $G$ be a domain in $\mathbb{C}^n$ containing the origin. Denote by $\mathcal{O}^1_0(G)$ the set of closed complex purely one-dimensional analytic sets in $G$ passing through the origin. Denote by $E(X)$ the Euclidean area of $X\in \mathcal{O}^1_0(G)$. The holomorphic radius rh$(G)$ of $G$ is defined as
$$
\mbox{rh}(G) = \inf\{\lambda >0 : \exists X\in \mathcal{O}^1_0(G), E(X) = \pi \lambda^2\}.
$$
\end{definition}

\begin{example}
\label{eg:Lelong}
Let $\mathbb{B}^4(r)$ be the Euclidean ball of $\mathbb{C}^2$ with radius $r$, then $\mbox{rh}(\mathbb{B}^4(r)) = r$. In fact the area $E(X)$ of $X\in\mathcal{O}^1_0(\mathbb{B}^4(r))$ is bounded from below by the area $\pi r^2$ of a section of the ball by a complex line through the origin (Lelong, 1950; see \cite{Bishop}).
\end{example}

The following theorem is known as Bishop's convergence theorem (see, for example, \cite{Bishop}), it will be used in the rest of the paper:

\begin{theorem}
\label{thm:bishop}
Let $\{A_j\}$ be a sequence of purely $p$-dimensional analytic subsets in a complex manifold $\Omega$ with locally uniformly bounded volumes:
$$
\mbox{\emph{Vol}}_{2p}(A_j\cap K)\leq M_K<\infty
$$
for any compact set $K\subset \Omega$. Here $M_K$ is a constant depending only on $K$. Then we can extract a subsequence from $\{A_j\}$ converging on compact subsets in $\Omega$ (in Hausdorff sense) to a purely $p$-dimensional analytic subset or to the empty set.
\end{theorem}

The following result is due to A. Sukhov and A. Tumanov \cite{AT_Symp_Nonsqueeze}, it provides a necessary condition on holomorphic radius for the existence of symplectic embedding. This result will be used in the proof of Theorem \ref{thm:wong2}.
\begin{theorem}
\label{thm:AT_nonsqueeze}
(\cite{AT_Symp_Nonsqueeze}) Let $G_1$ be a domain in $\mathbb{C}^2$ containing the origin and let $G_2$ be a domain in $\mathbb{D}(R)\times\mathbb{C}$ for some $R>0$. Assume there exists a symplectomorphism $\phi: G_1\rightarrow G_2$, then \emph{rh}$(G_1)\leq R$.
\end{theorem}

For $v=(v_1,\dots, v_4),w=(w_1,\dots, w_4)\in\mathbb{R}^4$, we denote the real inner product by $\langle v,w\rangle_{\mathbb{R}^4} = \sum^4_{j=1} v_j w_j$. Similarly for $v=(v_1,v_2), w=(w_1,w_2)\in\mathbb{C}^2$, we denote the complex inner product by $\langle v,w\rangle_{\mathbb{C}^2} = \sum^2_{j=1} v_j \overline{w_j}$. Notice that $\langle v,w\rangle_{\mathbb{R}^4} = \mbox{Re}\langle v,w\rangle_{\mathbb{C}^2}$.

By using the properties of inner product, the following lemma can be proved easily.
\begin{lemma}
\label{real-complex-perp}
\mbox{ }
\begin{enumerate}
\item Let $L\in\mathbb{C}^2$ be a real two dimensional plane. Denote by $L^{\perp_{\mathbb{R}^4}}$ the orthogonal complement of $L$ with respect to  $\langle \cdot,\cdot\rangle_{\mathbb{R}^4}$ and by $L^{\perp_{\mathbb{C}^2}}$ the orthogonal complement of $L$ with respect to $\langle \cdot,\cdot\rangle_{\mathbb{C}^2}$. If $L$ is a complex line, that is $v\in L$ if and only if $iv\in L$ for all $v\in\mathbb{C}^2$, then $L^{\perp_{\mathbb{R}^4}}=L^{\perp_{\mathbb{C}^2}}$.
\item If $L\in\mathbb{C}^2$ is a complex line, then $L^{\perp_{\mathbb{C}^2}}$ is also a complex line.
\end{enumerate}
\end{lemma}
%

We denote by $\mathfrak{I}$ the set consisting of four diagonal matrices:
$$
\mathfrak{I} = \left\{\begin{pmatrix}
1&&&\\
&a && \\
&&1&\\
&&&b
\end{pmatrix} \mbox{ : } a=\pm 1, b=\pm 1\right\}
$$

The following is the main theorem of this section. We used the canonical identification between complex matrices on $\mathbb{C}^2$ and real matrices on $\mathbb{R}^4$:

\begin{theorem}
\label{thm:wong2}
Let $T\in O(4)$ be an orthogonal transformation. $T\mathbb{D}^2$ is symplectomorphic to $\mathbb{D}^2$ with respect to the standard symplectic form on $\mathbb{R}^4$ if and only if there exists $U\in U(2)$ such that $UT\in\mathfrak{I}$.
\end{theorem}

\begin{proof}
$(\Leftarrow)$ Suppose there exists an $U\in U(2)$ such that $UT\in\mathfrak{I}$, then we know that $UT\mathbb{D}^2 = \mathbb{D}^2$ as a set. Furthermore $U\in U(2)$ is a linear symplectomorphism on $\mathbb{C}^2$. Hence $T\mathbb{D}^2$ is symplectomorphic to $\mathbb{D}^2$.

$(\Rightarrow)$ Let $(z_1, z_2)$ be the coordinate on $\mathbb{C}^2$. First of all, let $\partial\mathbb{D}^2\cap \partial \mathbb{B}^4(1) = S_1 \cup S_2$ where $S_1 = \{ |z_1| = 1, z_2 = 0\}$ and $S_2 = \{z_1 = 0, |z_2| = 1\}$. Therefore $S_1$ and $S_2$ are contained in the complex line $H_1=\{z_2 = 0\}$ and $H_2 = \{z_1 = 0\}$ respectively. For $i=1,2$, let $u_i, v_i\in\mathbb{C}^2$ be orthonormal basis of $TH_i$ under the real inner product $\langle\cdot,\cdot\rangle_{\mathbb{R}^4}$ on $\mathbb{R}^4$. Note that $TS_i$ can be parameterized by
$$
\frac{1}{2}\left(t+\frac{1}{t}\right)u_i + \frac{1}{2i}\left(t-\frac{1}{t}\right)v_i
$$
for $|t| = 1$ in $\mathbb{C}$. The complexification of $TS_i$, denoted by $\widetilde{TS_i}$, is given by the same parametrization but allowing $t\in\mathbb{CP}^1$. Here $\mathbb{CP}^n$ is the complex projective space of complex dimension $n$. Hence $\widetilde{TS_i}$ is a complex algebraic curve in $\mathbb{CP}^2$ parameterized by $t\in \mathbb{CP}^1$.

Notice that for $i = 1,2$, $\widetilde{TS_i}$ passes through the origin in $\mathbb{C}^2$ if and only if $u_i$ and $v_i$ are $\mathbb{C}$-dependent.

Suppose $T\mathbb{D}^2$ is symplectomorphic to $\mathbb{D}^2$, then Theorem~\ref{thm:AT_nonsqueeze} implies that rh$(T\mathbb{D}^2) \leq 1$. By Theorem \ref{thm:bishop} there exists $X\in \mathcal{O}^1_0(T\mathbb{D}^2)$ such that $E(X) = \pi (\mbox{rh}(T\mathbb{D}^2))^2$. Suppose there exists $p\in \partial X \cap \partial \mathbb{B}^4(1)$ such that $p\in \mbox{Int}(T\mathbb{D}^2)$; then $X$ is not entirely contained in $\mathbb{B}^4(1)$. Hence $E(X) > E(X \cap \mathbb{B}^4(1)) \geq \pi$ (see Example \ref{eg:Lelong}), which implies rh$(T\mathbb{D}^2) >1$, a contradiction. Therefore $\partial X\subset \partial \mathbb{B}^4(1)\cap \partial T\mathbb{D}^2 = TS_1 \cup TS_2$ and $X$ is a complex one dimensional analytic subset in $\mathbb{C}^2 \backslash (TS_1\cup TS_2)$. Since $TS_1\cup TS_2$ is a real one dimensional curve, it is totally real. Hence, by the reflection principle for analytic sets (see, for example, Section 20.5 of \cite{Bishop}), $X$ extends as a complex one dimensional analytic set to a neighborhood of $TS_1\cup TS_2$. By the uniqueness theorem $X$ is contained in the complex algebraic curve $\widetilde{TS_1}\cup\widetilde{TS_2}$.

Since $X$ contains the origin in $\mathbb{C}^2$, without loss of generality we can assume $\widetilde{TS_1}$ contains the origin. By the discussion above, we know that $u_1$ and $v_1$ are $\mathbb{C}$-dependent. Hence $TH_1 = \mbox{span}_{\mathbb{R}}\{u_1, v_1\} = \mbox{span}_{\mathbb{R}}\{u_1, iu_1\} = \mbox{span}_{\mathbb{C}}\{u_1\} = \widetilde{TS_1}$. This shows that $TH_1$ is a complex line.

By Lemma~\ref{real-complex-perp}, $H_2 = H_1^{\perp_{\mathbb{C}^2}} =  H_1^{\perp_{\mathbb{R}^4}}$. Since $T$ is an orthogonal matrix, we have $TH_2 = (TH_1)^{\perp_{\mathbb{R}^4}} = (TH_1)^{\perp_{\mathbb{C}^2}}$ where the last equality follows from Lemma~\ref{real-complex-perp} and the fact that $TH_1 = \widetilde{TS_1}=\mbox{span}_{\mathbb{C}}\{u_1\}$ is a complex line. Therefore Lemma~\ref{real-complex-perp} implies that $TH_2$ is a complex line.

We've shown that if $T$ is orthogonal and $T\mathbb{D}^2$ is symplectomorphic to $\mathbb{D}^2$, then $T$ maps the complex lines $H_1=\{z_2 = 0\}, H_2=\{z_1 = 0\}$ to complex lines $TH_1, TH_2$. Therefore there exists a unitary matrix $U\in U(2)$ such that $UT\in\mathfrak{I}$.
\end{proof}

\section{Symplectic rigidity in high dimensional case}
\label{section3}
Let $\mathbb{D}^m(r)=\{(z_1,\dots, z_m)\in\mathbb{C}^m : |z_j| < r \mbox{ for }j=1,\dots,m\}$ by the $m$-th product of discs of radius $r$. The following is the main theorem in this section:

\begin{theorem}
\label{thm:tri-disc}
For $r\geq 1$ and $n\geq 2$, the domains $\mathbb{D}_{\mathbb{R}}^2\times \mathbb{D}^{n-2}(r)$ and $\mathbb{D}^2\times\mathbb{D}^{n-2}(r)$ in $\mathbb{C}^n$ equipped with the standard symplectic form are not symplectomorphic.
\end{theorem}

We will first give the proof for the case $r>1$ by adapting the idea in the proof of Theorem 2.2 in \cite{Wong}. We will then develop a new method to prove Theorem \ref{thm:tri-disc} for the case $r=1$.

\subsection{The case $r>1$}
In the case $r>1$, theorem \ref{thm:tri-disc} follows from a more general result:
\begin{theorem}
Given $r>1$, for any real number $R>0$, if there exists a symplectic embedding $\phi:\mathbb{D}_{\mathbb{R}}^2\times \mathbb{D}^{n-2}(r) \rightarrow \mathbb{D}(R)\times\mathbb{C}^{n-1}$, then $R>1$.
\end{theorem}

\begin{proof}
For $R>0$, suppose there exists a symplectic embedding from $\mathbb{D}_{\mathbb{R}}^2\times \mathbb{D}^{n-2}(r)$ into $\mathbb{D}(R)\times\mathbb{C}^{n-1}$. It is proved in \cite{Wong} that for every $1\leq r_1<\frac{2}{\sqrt{\pi}}$, there is a symplectic embedding from $\mathbb{B}^4(r_1)$ into $\mathbb{D}^2_{\mathbb{R}}$. Take $1< r_1<\frac{2}{\sqrt{\pi}}$, then by combining these two embeddings, we obtain an embedding from $\mathbb{B}^{2n}(a)$ into $\mathbb{D}(R)\times\mathbb{C}^{n-1}$ where $a=\min(r, r_1)>1$. Therefore, by Gromov's non-squeezing theorem \cite{Gromov}, we have $R>1$.
\end{proof}

\subsection{The case $r = 1$}

In order to prove the case $r=1$, we need the following theorem regarding the existence of $J$-holomorphic discs, which is due to A. Sukhov and A. Tumanov \cite{AT_exist_Jdisc}. The original statement was about the triangular cylinder $\Delta\times \mathbb{C}^{n-1}$ where $\Delta = \{z\in\mathbb{C}: 0<\mbox{Im}z<1-|\mbox{Re}z|\}$ instead of the circular cylinder $\mathbb{D}\times \mathbb{C}^{n-1}$. However, one can see the result still holds for the circular cylinder by applying an area preserving map of the triangle to the disc.


\begin{theorem} (A. Sukhov and A. Tumanov \cite{AT_exist_Jdisc})
\label{thm:AT_exist_Jdisc}
Let $A$ be a continuous $n\times n$ matrix function on $\mathbb{C}^n$ with compact support in $\mathbb{D}\times\mathbb{C}^{n-1}$. Suppose there is a constant $0<a<1$ such that

\begin{equation}
\label{condition:complex_matrix}
\|A(z)\|\leq a, \forall z\in\mathbb{D}\times\mathbb{C}^{n-1}.
\end{equation}

Then there exists $p>2$ such that for every point $x\in \mathbb{D}\times\mathbb{C}^{n-1}$ there is a solution $Z\in W^{1,p}(\mathbb{D})$ (Sobolev space) of equation (\ref{eq:CauchyRiemann2})
$$
Z_{\overline{\zeta}} = A(Z)\overline{Z_{\overline{\zeta}}}
$$
such that $Z(\overline{\mathbb{D}})\subset\overline{\mathbb{D}\times\mathbb{C}^{n-1}}, x\in Z(\mathbb{D})$, \emph{Area}$(Z)=\pi$ and
$$
Z(\partial\mathbb{D})\subset\partial (\mathbb{D}\times\mathbb{C}^{n-1}) = (\partial\mathbb{D})\times\mathbb{C}^{n-1}.
$$
Furthermore, if we denote the components of $Z$ by $Z = (f_1, \dots, f_n)$, then we have the following area property
$$
\mbox{\emph{Area}}(f_1) = \pi, \mbox{\emph{Area}}(f_j) = 0, \mbox{ for }j=2,\dots,n.
$$
\end{theorem}

For $1\leq j\leq n$, let $M_j$ be the holomorphic disc $M_j=(m_1,\dots, m_n):\mathbb{D}\rightarrow\mathbb{C}^n$ where $m_k(z) = 0$ if $k\neq j$ and $m_j(z) = z$. Notice that the minimal area of an analytic set passing through the origin in $\mathbb{D}_{\mathbb{R}}^2\times \mathbb{D}^{n-2}$ is $\pi$, this is because $\mathbb{B}^{2n}\subset \mathbb{D}_{\mathbb{R}}^2\times \mathbb{D}^{n-2}$ and the minimal area of analytic set of $\mathbb{B}^{2n}$ passing through the origin is $\pi$ (Lelong 1950; see \cite{Bishop}).

\begin{lemma}
\label{lemma:n-2discs}
The minimal analytic set of $\mathbb{D}_{\mathbb{R}}^2\times \mathbb{D}^{n-2}$ through the origin is given by one of the $n-2$ distinct holomorphic discs $M_3, \dots, M_n$.
\end{lemma}

\begin{proof}
Let $S_1=\{x_1^2 + x_2^2 = 1, y_1=y_2=0, z_3=\dots=z_n=0\}, S_2=\{y_1^2 + y_2^2 = 1, x_1=x_2=0, z_3=\dots=z_n=0\}, S_j= \{|z_j| = 1, z_k = 0 \mbox{ for } k \neq j\} \mbox{ for } 3\leq j\leq n$. By using Lelong's result (see \cite{Bishop}) and the argument in proof of Theorem \ref{thm:wong2}, we conclude that the boundary of the analytic set $E$ of minimal area in $\mathbb{D}_{\mathbb{R}}^2\times \mathbb{D}^{n-2}$ through the origin must lie in the intersection of the boundary of $\mathbb{B}^{2n}$ and the boundary of $\mathbb{D}_{\mathbb{R}}^2\times \mathbb{D}^{n-2}$, notice that this intersection consists of $n$ circles $S_1, \dots, S_n$. Suppose a boundary point of $E$ lies in $S_1\cup S_2$, then $E$ must have a component lying in the complexification of $S_1\cup S_2$, which is given by $\{z_1^2+z_2^2=1, z_3=\cdots=z_n=0\}$, in fact all of $E$ lies in this set since $E$ is of minimal area. However $\{z_1^2+z_2^2=1, z_3=\dots=z_n=0\}$ does not pass through the origin, so the boundary of $E$ is contained in the circles $S_3\cup\dots\cup S_n$. Hence $E$ is one of the discs $M_3, \dots, M_n$.
\end{proof}

The following lemma is a consequence of Lemma \ref{lemma:n-2discs} and Theorem \ref{thm:bishop}:

\begin{lemma}
\label{lemma:limit_analytic_set}
Let $E_j$ be a convergent sequence of analytic sets in $\mathbb{D}_{\mathbb{R}}^2\times \mathbb{D}^{n-2}$ passing through the origin so that
$$
\lim_{j\rightarrow\infty}\mbox{\emph{Area}}(E_j) =\pi.
$$
Then the limiting analytic set $E_{\infty}$ is one of the $n-2$ distinct holomorphic discs $M_3, \dots, M_n$.
\end{lemma}

Our proof of Theorem \ref{thm:tri-disc} in the case $r=1$ is based on the fact that the domains $\mathbb{D}_{\mathbb{R}}^2\times \mathbb{D}^{n-2}$ and $\mathbb{D}^{n}$ have different number of analytic sets of minimum area through the origin. We are now ready to prove the main theorem in this section.

\begin{theorem}
The domains $\mathbb{D}_{\mathbb{R}}^2\times \mathbb{D}^{n-2}$ and $\mathbb{D}^{n}$ equipped with the standard symplectic form on $\mathbb{C}^n$ are not symplectomorphic.
\end{theorem}

\begin{proof}
Suppose on the contrary that $\psi:\mathbb{D}_{\mathbb{R}}^2\times \mathbb{D}^{n-2}\rightarrow \mathbb{D}^{n}$ is a symplectomorphism. By composing a symplectomorphism of $\mathbb{D}^n$, we can assume that $\psi(0) = 0$.

Consider the standard almost complex structure $J_{\mbox{st}}$ on $\mathbb{D}_{\mathbb{R}}^2\times \mathbb{D}^{n-2}$ and let $J = \psi_*J_{\mbox{st}}$ be the complex structure on $\mathbb{D}^{n}$ given by the push-forward of $J_{\mbox{st}}$ by $\psi$. Since $\psi
^*\omega_{\mbox{st}} =\omega_{\mbox{st}}$, the almost complex structure $J$ is tamed by $\omega_{\mbox{st}}$. Then the complex matrix $\widetilde{A}$ of $J$ satisfies $\|\widetilde{A}(z)\|<1$ for $z\in\mathbb{D}^{n}$.

Let $\{K_l\}_{l=1}^{\infty}$ be a compact exhaustion of $\mathbb{D}^n$ so that each $K_l$ is a closed polydisc with radius less than 1, that is, $K_l\subset K_{l+1}$, $K_l$ is a compact subset of $\mathbb{D}^n$ for all $l$ and $\cup_{l=1}^{\infty}K_l = \mathbb{D}^n$. For each $l$, let $\chi_l$ be a smooth cut-off function on $\mathbb{C}^{n}$ with support in $\mathbb{D}^n$ and equal to 1 on $K_l$. Define $A_l = \chi_l\widetilde{A}$ to be a $n\times n$ matrix function on $\mathbb{C}^n$ such that $A_l=0$ outside $\mathbb{D}^n$. Since $\|\widetilde{A}\| < 1$ on $\mathbb{D}^n$, there is a constant $0<a<1$ such that (\ref{condition:complex_matrix}) holds for $A_l$. Let $J_l$ be the almost complex structure on $\mathbb{C}^{n}$ corresponding to the complex matrix $A_l$.

By considering $\mathbb{D}^n$ as a subset of $\mathbb{D}\times\mathbb{C}^{n-1}$, we can apply Theorem \ref{thm:AT_exist_Jdisc} so that for each $l$, there exists a $J_l$-holomorphic disc $f_l:\mathbb{D}\rightarrow\mathbb{D}\times\mathbb{C}^{n-1}$ such that the image of $f_l$ passes through the origin. Also if we write $f_l = (f_{l,1},\dots,f_{l,n})$, then we have Area$(f_{l,j}) = \delta_{j1}\pi$ for all $l$, here $\delta_{j1}$ is the Kronecker delta.

Fix an integer $N$, for each $l\geq N$, $\psi^{-1}(f_l(\mathbb{D})\cap K_N)$ is an analytic set in $\psi^{-1}(K_N)\subset\mathbb{D}_{\mathbb{R}}^2\times \mathbb{D}^{n-2}$ passing through the origin. Since $\psi$ is a symplectomorphism, we have
$$
\mbox{Area}(\psi^{-1}(f_l(\mathbb{D})\cap K_N)) \leq \mbox{Area}(f_l(\mathbb{D})\cap K_N)\leq \pi.
$$
Therefore by Theorem \ref{thm:bishop}, after passing to a subsequence,
$$
F_N = \lim_{l\rightarrow\infty}\psi^{-1}(f_l(\mathbb{D})\cap K_N)
$$
exists and Area$(F_N)\leq \pi$. Notice that $F_N$ is not an empty set for $N$ sufficiently large, this is because $0\in\psi^{-1}(f_l(\mathbb{D})\cap K_N)$ for all $l\geq N$.



The above argument holds for all $N$, so we can apply Theorem \ref{thm:bishop} again to the sequence of analytic set $F_N$ as $N\rightarrow\infty$. After passing to a subsequence, denote the limit of $F_N$ by $F$. Now $F$ is an analytic set in $\mathbb{D}_{\mathbb{R}}^2\times \mathbb{D}^{n-2}$ passing through the origin with Area$(F)\leq \pi$ and $\partial F\subset \partial(\mathbb{D}_{\mathbb{R}}^2\times \mathbb{D}^{n-2})$. Since the minimal area of analytic set in $\mathbb{D}_{\mathbb{R}}^2\times \mathbb{D}^{n-2}$ through the origin is $\pi$, so we must have Area$(F)=\pi$. Therefore $F$ is one of the holomorphic discs $M_j$ for $3\leq j\leq n$ by Lemma \ref{lemma:limit_analytic_set}.

Let $E=\psi(F)$. We now know that Area$f_l = \pi$ for all $l$ and $f_l(\mathbb{D})\cap \mathbb{D}^n\rightarrow E$ as $l\rightarrow\infty$, also we have Area$(E)=\pi$. We want to show that $f_l(\mathbb{D})\rightarrow E$ as $l\rightarrow\infty$. Let $X_l = f_l(\mathbb{D})\setminus\mathbb{D}^n$, that is the image of $f_l$ which is not in $\mathbb{D}^n$. By the construction of $A_l$ and $J_l$, we know that $J_l = J_{\mbox{st}}$ outside $\mathbb{D}^n$, hence $X_l$ is an usual analytic set in $(\mathbb{D}\times \mathbb{C}^{n-1})\setminus\mathbb{D}^n$. Since Area$X_l\leq$ Area$f_l=\pi$ for all $l$, we can apply Theorem \ref{thm:bishop} to conclude that, after passing to a subsequence, $X_l$ converges to an analytic set $X$. However $f_l(\mathbb{D})\cap \mathbb{D}^n\rightarrow E$ as $l\rightarrow\infty$ and Area$(E)=\pi$ implies that
$$
\lim_{l\rightarrow\infty}\mbox{Area}(f_l(\mathbb{D})\cap \mathbb{D}^n) = \pi,
$$
and by construction Area$(f_l)=\pi$ for all $l$, hence we have Area$(X_l)\rightarrow 0 $ as $l\rightarrow\infty$. Therefore $X$ is an empty set and we can conclude that
$$
\lim_{l\rightarrow\infty}f_l(\mathbb{D})\subset\mathbb{D}^n,
$$
and hence
$$
\lim_{l\rightarrow\infty}f_l(\mathbb{D})=E.
$$
Since Area$(f_{l,j})=\delta_{j1}\pi$, if we write $\omega_{\mbox{st}} = \omega_1 +\cdots +\omega_n$ where $\omega_j = dx_j\wedge dy_j$ for $j = 1,\dots, n$, then we have
$$
\int_E\omega_j = \delta_{j1}\pi.
$$

Now for $1\leq k\leq n$, by considering $\mathbb{D}^n$ as a subset of the cylinder $\mathbb{C}^{k-1}\times\mathbb{D}\times\mathbb{C}^{n-k}\cong\mathbb{D}\times\mathbb{C}^{n-1}$, we can apply the above argument to obtain a real 2-dimensional set $E_k$ in $\mathbb{D}^n$ passing through the origin, satisfying the following conditions:
\begin{enumerate}
\item $\int_{E_k}\omega_j = \delta_{jk}\pi$ for $j = 1,\dots, n$, hence all $E_k$ are distinct for $1\leq k\leq n$.
\item The preimage $F_k=\psi^{-1}(E_k)$ is an analytic set in $\mathbb{D}_{\mathbb{R}}^2\times \mathbb{D}^{n-2}$ passing through the origin.
\item $F_k$ are distinct analytic sets for $1\leq k\leq n$ since $E_k$'s are distinct and $\psi$ is a bijection.
\item Area$(F_k)= \pi$ for $1\leq k\leq n$.
\end{enumerate}
Hence for each $1\leq k\leq n$, $F_k$ must be one of the holomorphic discs $M_j$ for $3\leq j\leq n$ according to Lemma \ref{lemma:limit_analytic_set}, but this is impossible since all $F_k$'s are distinct, so we arrived at a contradiction. Therefore $\mathbb{D}_{\mathbb{R}}^2\times \mathbb{D}^{n-2}$ and $\mathbb{D}^{n}$ equipped with the standard symplectic form on $\mathbb{C}^n$ are not symplectomorphic.

\end{proof}

\noindent\textbf{Acknowledgments.} The author is grateful to A. Tumanov for formulating this problem and providing useful suggestions.
%
%



%
%

 
\end{document}